\documentclass[12pt]{amsart}
\usepackage{amsmath,amsthm,verbatim,amssymb,amsfonts,amscd, graphicx,enumitem}
\usepackage{graphics}
\usepackage{tikz-cd}
\usepackage{pinlabel}
\usepackage{xcolor}
\usepackage{stackrel}
\usepackage{caption}
\usepackage[labelformat=simple]{subcaption}
\usepackage{hyperref}

\usepackage[pagewise]{lineno}
\usepackage[margin=1.5in]{geometry}
\newtheorem{theorem}{Theorem}[section]
\newtheorem{corollary}[theorem]{Corollary}
\newtheorem{lemma}[theorem]{Lemma}
\newtheorem{proposition}[theorem]{Proposition}

\newtheorem*{cor:euler}{Corollary \ref{cor:euler}}
\newtheorem*{cor:orient0}{Corollary \ref{cor:orient0}}

\newtheorem*{cor:differentbridge}{Corollary \ref{cor:differentbridge}}

\makeatletter
\newtheorem*{rep@theorem}{\rep@title}
\newcommand{\newreptheorem}[2]{%
\newenvironment{rep#1}[1]{%
 \def\rep@title{#2 \ref{##1}}%
 \begin{rep@theorem}}%
 {\end{rep@theorem}}}
\makeatother

\newtheorem{question}[theorem]{Question}

\theoremstyle{definition}

\newtheorem{remark}[theorem]{Remark}

\newtheorem*{observation*}{Observation}

\newreptheorem{theorem}{Theorem}
\newreptheorem{lemma}{Lemma}
\newreptheorem{question}{Question}
\newreptheorem{corollary}{Corollary}
\newreptheorem{proposition}{Proposition}

\theoremstyle{definition}

\newcommand{\leftrarrows}{\mathrel{\raise.75ex\hbox{\oalign{%
  $\scriptstyle\leftarrow$\cr
  \vrule width0pt height.5ex$\hfil\scriptstyle\relbar$\cr}}}}
\newcommand{\lrightarrows}{\mathrel{\raise.75ex\hbox{\oalign{%
  $\scriptstyle\relbar$\hfil\cr
  $\scriptstyle\vrule width0pt height.5ex\smash\rightarrow$\cr}}}}
\newcommand{\Rrelbar}{\mathrel{\raise.75ex\hbox{\oalign{%
  $\scriptstyle\relbar$\cr
  \vrule width0pt height.5ex$\scriptstyle\relbar$}}}}

\makeatletter
\def\leftrightarrowsfill@{\arrowfill@\leftrarrows\Rrelbar\lrightarrows}
\newcommand{\xleftrightarrows}[2][]{\ext@arrow 3399\leftrightarrowsfill@{#1}{#2}}
\makeatother

\DeclareMathOperator{\Int}{Int}

\newcommand{\Z}{\mathbb{Z}}

\newcommand{\DD}{\mathbb{D}}

\newcommand{\N}{\mathbb{N}}
\newcommand{\TT}{\mathfrak{T}}
\newcommand{\Ss}{\mathcal{S}}
\newcommand{\D}{\mathcal{D}}

\newcommand{\K}{\mathcal{K}}

\newcommand{\T}{\mathcal{T}}

\newcommand{\pd}{\partial}
\newcommand{\be}{\begin{enumerate}}
\newcommand{\ee}{\end{enumerate}}

\newcommand{\A}{\alpha}
\newcommand{\n}{\beta}

\newcommand{\eps}{\varepsilon}
\newcommand{\X}{\times}
\newcommand{\Ff}{\mathcal{F}}

\let\int\relax
\newcommand{\int}{\mathring}

\usepackage{color}
\definecolor{darkpurple}{rgb}{.5,0,.5}
\definecolor{darkgreen}{rgb}{0,.5,0}

\begin{document}

\title{Bridge trisections and Seifert solids}

\author[Joseph]{Jason Joseph}
\address{Department of Mathematics\\Rice University\\Houston, TX, USA}
\email{jason.joseph@rice.edu}

\author[Meier]{Jeffrey Meier}
\address{Department of Mathematics\\Western Washington University\\Bellingham, WA, USA}
\email{jeffrey.meier@wwu.edu}

\author[Miller]{Maggie Miller}
\address{Department of Mathematics\\Stanford University\\Stanford, CA, USA}
\email{maggie.miller.math@gmail.com}

\author[Zupan]{Alexander Zupan}
\address{Department of Mathematics\\University of Nebraska-Lincoln\\Lincoln, NE, USA}
\email{zupan@unl.edu}


\begin{abstract}
	We adapt Seifert's algorithm for classical knots and links to the setting of tri-plane diagrams for bridge trisected surfaces in the 4--sphere.
	Our approach allows for the construction of a Seifert solid that is described by a Heegaard diagram.
	The Seifert solids produced can be assumed to have exteriors that can be built without 3--handles; in contrast, we give examples of Seifert solids (not coming from our construction) whose exteriors require arbitrarily many 3--handles.
	We conclude with two classification results.
	The first shows that surfaces admitting doubly-standard shadow diagrams are unknotted.
	The second says that a $b$--bridge trisection in which some sector contains at least $b-1$ patches is completely decomposable, thus the corresponding surface is unknotted.
	This settles affirmatively a conjecture of the second and fourth authors.
\end{abstract}

\maketitle

\section{Introduction}\label{sec:intro}

A {\emph{bridge trisection}} of a surface $\Ss$ in $S^4$ is a certain decomposition of $(S^4,\Ss)$ into three trivial disk systems $(B^4_1,\D_1),(B^4_2,\D_2),(B^4_3,\D_3)$ that can be encoded diagrammatically either as a triple of tangles called a \emph{tri-plane} diagram or as a corresponding \emph{shadow diagram}.
In this paper, we show how topological information about the surfaces $\Ss$ can be recovered from these diagrammatic representations.

In Section \ref{sec:seifert}, we give a version of Seifert's algorithm for bridge-trisected surfaces, showing how a tri-plane diagram can be used to produce a 3--manifold bounded by a connected surface $\Ss$ with normal Euler number zero.

\begin{reptheorem}{proc}
	If $\Ss$ is connected and $e(\Ss) = 0$, then there is a procedure to produce a Seifert solid for $\Ss$ that takes as input a tri-plane diagram for $\Ss$.
\end{reptheorem}

In Subsection \ref{subsec:getheegaard}, we give an explicit procedure for constructing a Heegaard diagram for such a 3--manifold when $\Ss\cong S^2$. As a corollary of the work in building Seifert solids, we recover a combinatorial proof of the existence of Seifert solids. We also show that certain bridge trisected surfaces are unknotted.

\begin{reptheorem}{cor:doublestandard}
	If a surface $\Ss$ has a doubly-standard shadow diagram, then $\Ss$ is unknotted.
\end{reptheorem}

In Section~\ref{sec:afss}, we give 4--dimensional analogs to the 3--dimensional concepts of free Seifert surfaces and canonical Seifert surfaces.
We call a Seifert solid \emph{canonical} if it is obtained from the procedure presented in Section~\ref{sec:seifert}, and we call a Seifert solid \emph{spinal} if its exterior in $S^4$ can be built without 3--handles.
We prove the following two results relating (and distinguishing) these concepts.

\begin{reptheorem}{thm:free}
	If a surface-knot $\Ss$ admits a Seifert solid, then it admits a canonical Seifert solid that is spinal.
\end{reptheorem}

In fact, modulo some additional, easily satisfied connectivity conditions, every canonical Seifert solid is spinal. The next result shows that some Seifert solids (in contrast to canonical Seifert solids and many standard examples) are ``far'' from being spinal.

\begin{reptheorem}{thm:not_free}
	Given any $n\in\N$, there exists a 2--knot $\K$ that bounds a Seifert solid $Y$ homeomorphic to $(S^1\times S^2)^\circ$ such that $S^4\setminus\nu(Y)$ requires at least $n$ 4--dimensional 3--handles.
\end{reptheorem}

Finally, in Section \ref{sec:c=b} we prove the following standardness result, affirmatively settling Conjecture~4.3 of~\cite{MeiZup_bridge1}.  

\begin{reptheorem}{thm:c=b}
	Let $\TT$ be a $(b;{\mathbf{c}})$--bridge trisection with $c_i=b_i-1$ for some $i\in\Z_3$.  Then, $\TT$ is completely decomposable, and the underlying surface-link is either the unlink of $\min\{c_i\}$ 2--spheres or the unlink of $\min\{c_i\}$ 2--spheres and one projective plane, depending on whether $|c_{i-1}-c_{i+1}|=1$ or $0$.
\end{reptheorem}

 The proof relies on theorems of Scharlemann and Bleiler-Scharlemann regarding planar surfaces in 3--manifolds~\cite{BleSch_proj,Sch_4crit}. The second and fourth authors previously handled this case when $c_i=b$ for some $i\in\Z_3$~\cite[Proposition~4.1]{MeiZup_bridge1}. 

\subsection*{Acknowledgements}

This paper began following discussions at the workshop \emph{Unifying 4--Dimensional Knot Theory}, which was hosted by the Banff International Research Station in November 2019, and the authors would like to thank BIRS for providing an ideal space for sparking collaboration. We are grateful to Masahico Saito for sharing his interest in adapting Seifert's algorithm to bridge trisections and motivating this paper. The authors would like to thank Rom\'an Aranda, Scott Carter, and Peter Lambert-Cole for helpful conversations.  JJ was supported by MPIM during part of this project, as well as NSF grants DMS-1664567 and DMS-1745670. JM was supported by NSF grants DMS-1933019 and DMS-2006029. MM was supported by MPIM during part of this project, as well as NSF grants DGE-1656466 (at Princeton) and DMS-2001675 (at MIT) and a research fellowship from the Clay Mathematics Institute (at Stanford).  AZ was supported by MPIM during part of this project, as well as NSF grants DMS-1664578 and DMS-2005518.

\section{Preliminaries}\label{sec:prelim}

We work in the smooth category.  This section includes an abbreviated introduction to the concepts relevant to this paper, but the interested reader is encouraged to consult the reference \cite{gaykirby} for further information about 4--manifold trisections and the references \cite{MeiZup_bridge1} and \cite[Section~2]{JMMZ} for more detailed discussions of bridge trisections.  We limit our work here to surfaces in $S^4$, but there is also a theory of bridge trisections in arbitrary 4--manifolds; see~\cite{MeiZup_bridge2}.

\subsection{Bridge trisections}

Let $\Ss$ be an embedded surface in $S^4$, let $b$ be a positive integer, and let $\mathbf{c} = (c_1,c_2,c_3)$ be a triple of positive integers.  A \emph{$(b;\mathbf{c})$--bridge trisection} of $(S^4,\Ss)$ is a decomposition
\[ (S^4,\Ss) = (X_1,\D_1) \cup (X_2,\D_2) \cup (X_3,\D_3)\]
such that
\begin{enumerate}
\item Each $\D_i$ is a collection of $c_i$ boundary-parallel disks in the 4--ball $X_i$;
\item Each intersection $\T_i = \D_{i-1} \cap \D_i$ a boundary-parallel tangle in the 3--ball $H_i = X_{i-1} \cap X_i$ (with indices considered mod 3);
\item The triple intersection $\D_1 \cap \D_2 \cap \D_3$ is a collection of $b$ points in the 2--sphere $\Sigma = X_1 \cap X_2 \cap X_3$.
\end{enumerate}
In~\cite{MeiZup_bridge1}, it was proved that every surface $\Ss$ admits a $(b;\mathbf{c})$--bridge trisection for some $(b;\mathbf{c})$.  We choose orientations so that $\partial(X_i,\D_i) = (H_i,\T_i) \cup (\overline{H_{i+1}},\overline{\T_{i+1}})$.  When we wish to be succinct, we use $\TT$ to represent a bridge trisection, with components labeled as above.

\subsection{Diagrams for bridge trisections}

The existence of bridge trisections gives rise to a new diagrammatic theory for surfaces in $S^4$, using an object called a \emph{tri-plane diagram}, a triple $\DD = (\DD_1,\DD_2,\DD_3)$ of trivial planar diagrams with the additional condition that each $\DD_i \cup\overline\DD_{i+1}$ is a classical diagram for an unlink.  In~\cite{MeiZup_bridge1}, it was shown that every tri-plane diagram determines a bridge trisection $\TT$.  Conversely, given a bridge trisection $\TT$ of $(S^4,\Ss)$, we can choose a triple of disks $E_i \subset H_i$ with common boundary and project the tangles $\T_i$ onto $E_i$ to obtain a tri-plane diagram.  Of course, the choices of disks and projections are not unique, but any two tri-plane diagrams corresponding the same bridge trisection $\TT$ are related by a finite collection of \emph{interior Reidemeister moves} and \emph{mutual braid transpositions}, while any two bridge trisections $\TT$ and $\TT'$ for the same surface $\Ss$ are related by \emph{perturbation} and \emph{deperturbation} moves.

In addition, bridge trisections yield another type of diagram:  Each trivial tangle $\T_i$ can be isotoped rel-boundary into the surface $\Sigma$, yielding a triple $(A,B,C)$ of pairwise disjoint collections of arcs called a \emph{shadow diagram}, which has the property that $\pd A = \pd B = \pd C$, and the pairwise unions of any two of the tangles $\T_A$, $\T_B$, $\T_C$ determined by the arcs are unlinks.  As with tri-plane diagrams, any shadow diagram determines a bridge trisection.  Further details about shadow diagrams can be found in~\cite{MTZ_graph}.

Here we consider special types of shadow diagrams.  We say that a pair of collections of arcs in a shadow diagram is \emph{standard} if their union is embedded.  Any bridge trisection admits a shadow diagram $(A,B,C)$ in which one of the pairs is standard. If two or three pairs of shadows in a shadow diagram $(A,B,C)$ are standard, then we say that $(A,B,C)$ is {\emph{doubly-standard}} or {\emph{triply-standard}}, respectively.  Theorem~\ref{cor:doublestandard} says that doubly-standard (and thus triply-standard) diagrams always describe unknotted surfaces.

\subsection{Unknotted surfaces}
\label{subsec:unknotted}

In this subsection, we review standard notions of unknottedness for surfaces in $S^4$. A closed, connected, surface $\Ss$ in $S^4$ is \emph{unknotted} if it bounds an embedded 3--dimensional handlebody $H \subset S^4$.  For nonorientable surfaces, the definition is slightly more involved.  We define the two unknotted projective planes, $P_{\pm}$, to be the two standard projective planes in $S^4$, pictured via their tri-plane diagrams in Figure~\ref{fig:unknottedrp2}, where $e(P_{\pm}) = \pm 2$.

\begin{figure}[htp]
	\centering
	\includegraphics[width=.8\linewidth]{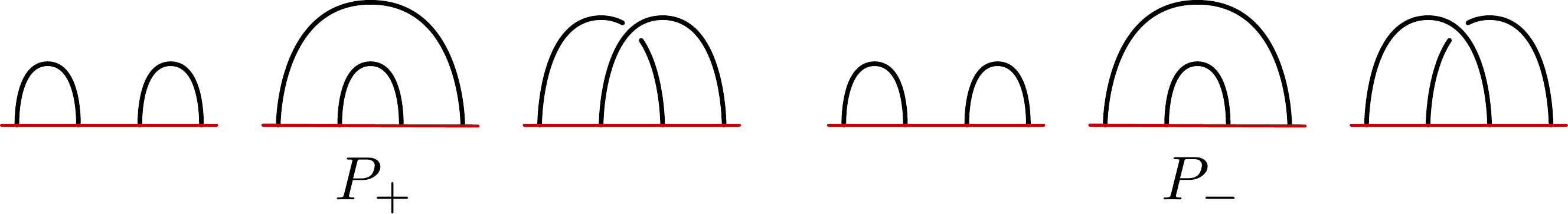}
	\caption{Tri-plane diagrams for $P_+$ and $P_-$.}
	\label{fig:unknottedrp2}
\end{figure}

In general, for a nonorientable surface $\Ss$, we say that $\Ss$ is \emph{unknotted} if $\Ss$ is isotopic to a connected sum of some number of copies of $P_+$ and $P_-$.  See~\cite[Remark 2.6]{JMMZ} for a detailed discussion of the orientation conventions used here.

\section{Seifert solids}\label{sec:seifert}

Classical results of Gluck~\cite{Gluck62} (resp., Gordon-Litherland~\cite{gordonlitherland}) assert that every orientable surface $\Ss$ (resp., surface $\Ss$ with $e(\Ss) = 0$) in $S^4$ bounds an embedded 3--manifold, called a \emph{Seifert solid} in the orientable case.  In the setting of broken surface diagrams, Carter and Saito provided a procedure that in many respects mimics Seifert's algorithm for classical knots~\cite{SeifertAlgorithm}.  In this section, we describe an extension of Seifert's algorithm that takes an oriented tri-plane diagram $\DD$ and produces a Seifert solid whose intersection with $\pd X_i$ agrees with the classical Seifert's algorithm performed on the oriented unlink diagram $\DD_i \cup \overline\DD_{i+1}$.  We also obtain alternative proofs of the theorems of Gluck and Gordon-Litherland mentioned above.

\subsection{Existence of Seifert solids}
\label{subsec:existence}

Given a spanning surface $F$ for an unlink $U$, we define the \emph{cap-off} $\mathcal{F}$ of $F$ to be the closed surface $\mathcal{F} \subset S^4$ obtained by gluing a collection of trivial disks in $B^4_-$ to $F$ along $U$. (There is a unique such choice of disks up to isotopy rel-boundary in $B^4_-$ by e.g. \cite{kss} or \cite{liv}.)  Let $F_+ \subset S^3$ denote the M\"obius band bounded by the unknot so that $F_+$ contains a positive half-twist and has boundary slope $+2$, and let $F_- \subset S^3$ denote the M\"obius band bounded by the unknot with a negative half-twist and boundary slope $-2$.  For $n > 0$, let $F_n$ be the connected surface obtained by attaching $n-1$ trivial bands to the split union of $n$ copies of $F_+$; that is, $F_n$ is obtained by taking the boundary connected sum of $n$ copies of $F_+$.  For $n < 0$, let $F_n$ be obtain by taking the boundary connected sum of $(-n)$ copies of $F_-$. Finally, let $F_0$ be the disk bounded by the unknot in $S^3$.  Additionally, let $\Ff_n$ be the cap-off of $F_n$.  In Figure \ref{fig:unknottedrp2}, the negative M\"obius band is shown to cap off into $B^4_+$ to obtain $P_+$. (See also~\cite[Figure~2]{JMMZ}.)  Here, we are capping off into $B^4_-$, so that by definition the cap-off $\mathcal{F}_{-1}$ of the negative M\"obius band $F_-$ is $P_-$.  In contrast, the cap-off $\mathcal{F}_1$ of the positive M\"obius band $F_+$ is $P_+$.  (Recall that $P_+$ and $P_-$ denote the two unknotted projective planes in $S^4$; see Subsection~\ref{subsec:unknotted}.)  It follows that
$$\Ff_n=
\begin{cases}\text{a connected sum of $n$ copies of $P_+$}, & \text{ if } n>0\\\text{a connected sum of $-n$ copies of $P_-$} & \text{ if } n<0\\\text{an unknotted 2--sphere}& \text{ if } n=0
\end{cases},$$
The intent of the cap-off notation is the emphasize the way in which $\Ff_n$ can be obtained from a specific surface in $S^3$, which will be useful in the rest of this section -- especially given the following lemma.

\begin{lemma}\label{compress}
Every incompressible spanning surface $F$ for the unknot is isotopic to $F_n$ for some $n \in \Z$.
\end{lemma}

\begin{proof}
First, we argue that $F_n$ is incompressible for all $n$.  This follows from \cite{tsau}, but we include a proof here.  Certainly, $F_0$ and $F_{\pm 1}$ are incompressible, since a compression increases Euler characteristic by two.  Suppose now that $F_n$ is compressible for some $n > 1$, and let $F_n'$ be the component of the surface obtained by compressing $F_n$ such that $\pd F_n' = \pd F_n$.  In addition, let $\Ff_n' \subset S^4$ be the cap-off of $F_n'$.  Then the embedded surface $\Ff_n$ can be obtained by from $\Ff_n'$ by a 1--handle attachment, and thus $e(\Ff_n') = e(\Ff_n) = 2n$.  However, since the nonorientable genus of $\Ff_n'$ is strictly less than $n$, this contradicts the Whitney--Massey Theorem (see discussion in \cite{JMMZ}).  We conclude that $F_n$ is incompressible.

On the other hand, suppose that $F$ is an arbitrary incompressible spanning surface for the unknot $U$.  The exterior of $U$ is a solid torus $V$, and every simple closed curve $c \subset \pd V$ is homotopic to a $(p,q)$--curve, where a $(0,1)$--curve is the boundary of a meridian disk of $V$ and a $(1,0)$--curve is the boundary of a meridian disk of $N(U)$. The boundary of $F$ is a $(2k,1)$--curve for some integer $k$. (The spanning surface $F$ intersects the disk bounded by $U$ in some number of arcs, the endpoints of which correspond to the intersections of the $(p,q)$--curve with the $(0,1)$--curve.) If $F$ is orientable, then it is well-known that $F$ is isotopic to the meridian disk $F_0$.  

Suppose that $F$ is nonorientable.  By~\cite[Corollary 12]{tsau}, the nonorientable genus of $F$ is equal to $|k|$.  Assuming that $\pd F$ and $\pd F_0$ meet efficiently, isotope $F$ so that it intersects $F_0$ minimally.  By standard cut-and-paste arguments, an arc of $F \cap F_0$ which is outermost in $F_0$ gives rise to a boundary-compressing disk $\Delta$ for $F$. Since $\pd F$ and $\pd F_0$ meet efficiently, the result $F'$ of boundary-compressing $F$ along $\Delta$ has a single boundary component and nonorientable genus $k-1$.  Reversing the process, we see that $F$ can be obtained from $F'$ by attaching a boundary-parallel band to $F'$ along opposite sides of $\pd F'$.  Note that $\pd V \setminus \pd F'$ is an annulus and the band is determined by a spanning arc.  Working rel-boundary, all choices of spanning arcs are related by Dehn twists about $\pd F'$, and so it follows that up to isotopy, there is a unique band taking $F'$ to $F$.

Finally, we claim that $F$ is isotopic to $F_k$, and we prove this fact by inducting on $k$.  If $k = \pm 1$, then $F$ has genus one and is obtained from the disk $F'=F_0$ by a single boundary tubing.  By the above argument, there is precisely one way to do this, and thus $F = F_{\pm 1}$.  Now, suppose that $k > 1$ and the claim holds for $j = k-1$.  As above, isotope $F$ to meet $F_0$ minimally, and since $k > 1$, there are at least two arcs $a_0$ and $a_1$ of $F \cap F_0$ that are outermost in $F_0$.  Let $\ell$ be a $(0,1)$-curve that meets $\pd F$ in a single point contained in $a_0$.  Then, $a_1$ gives rise to a boundary-compressing disk $\Delta_1$ and the result $F'$ of boundary-compressing $F$ along $\Delta_1$ also satisfies $|\pd F' \cap \ell| = 1$, since the modification was carried out away from the arc $a_0$.  We conclude that $F'$ has genus $k-1$ and boundary slope $(2(k-1),1)$.  By induction $F' = F_{k-1}$, and since there is a unique way to obtain $F$ from $F'$ by boundary-tubing, it follows that $F = F_k$.  The case $k < -1$ follows symmetrically, completing the proof of the lemma.
\end{proof}

In the next proposition, we use Lemma~\ref{compress} to understand the cap-off of any spanning surface $F$ for an unlink in $S^3$.

\begin{proposition}\label{spancap}
Let $F$ be a spanning surface for an unlink $U$ in $S^3$.
\be
\item If every component of $\pd F$ has slope $0$, then the cap-off $\Ff$ bounds a (possibly nonorientable, possibly disconnected) handlebody $V \subset B^4$ such that $V \cap \pd B^4 = F$.
\item\label{signitem} The normal Euler number $e(\Ff)$ is equal to the sum of the slopes of the boundary components of $F$.  
\item The cap-off $\Ff$ is a split union of unknotted surfaces in $S^4$.
\ee
\end{proposition}

\begin{proof}
	Suppose $F$ and $F'$ are two spanning surfaces for an unlink $U$ in $S^3$ such that $F'$ is isotopic relative to $U$ to the surface obtained by surgering $F$ along a compressing disk $D$ for $F$.
	Then there is a compression body $C\subset S^3\times[0,1]$ such that
	\begin{itemize}
		\item $C\cap(S^3\times\{1\})=F\times\{1\}$, 
		\item $C\cap(S^3\times\{0\})=F'\times\{0\}$, and 
		\item $\partial C = (F\times\{1\})\cup(\overline{F'\times\{0\}})\cup(U\times[0,1])$,
	\end{itemize}
	and $C$ has a single critical point (of index 1) with respect to the Morse function $S^3\times[0,1]\to[0,1]$, which we assume lies in $S^3\times\left\{\frac{1}{2}\right\}$.
	Note that $C$ is a product cobordism above and below $S^3\times\left\{\frac{1}{2}\right\}$.
	
	Any spanning surface $F$ for $U$ can be reduced to $F'$, a union of 2--spheres and incompressible spanning surfaces for components of $U$ via a sequence of compressions and isotopies.
	If each component of $\partial F$ has slope~0, then $F'$ is a collection of disks and spheres.
	Applying the compression body construction described above for each compression taking $F$ to $F'$ and stacking the results, we get a compression body $C$ co-bounded by $F$ and $F'$.
	Since $F'$ is a collection of disks and spheres, there is a handlebody with boundary $\Ff = F\cup \D$, where $\D = \overline{F'}\cup(U\times[0,1])$ is a collection of properly embedded disks in $B^4$: simply cap-off the sphere components of $C$ with 3--balls whose interiors are pushed sufficiently deep into $B^4$.
	This handlebody is non-orientable (resp., disconnected) if and only if $F$ is.
	This establishes part (1).

	Let $F$ be any spanning surface for an unlink $U = \bigsqcup_{i=1}^nU_i$.
	Let $B = \bigsqcup_{i=1}^nB_i$ be a collection of disjoint 3--balls with $U_i\subset\Int(B_i)$.
	Let $F' = \bigsqcup_{i=1}^nF_i$ be a split union of incompressible spanning surfaces for the components of $U$, with $F_i\subset\Int(B_i)$, so that the slopes of $F$ and $F'$ agree at each component of $U$.
	Let $F''$ be the result of surgering $F'$ along a collection of arcs so that $F''$ and $F$ have the same homeomorphism type relative to $U$; moreover, assume that every arc of the collection intersects each component of $\partial B$ in at most one point.
	It follows that $F''$ decomposes as a split union of connected sums of surfaces, each summand of which is either a torus or an incompressible spanning surface for an unknot.
	Therefore, the cap-off $\Ff''$ is the split union of connected sums of surfaces, each summand of which is an unknotted surface in $S^4$.
	Livingston showed that $F$ and $F''$ are isotopic rel-boundary in $B^4$~\cite{liv}.
	It follows that the cap-off $\Ff$ will isotopic to the cap-off $\Ff''$, which completes the proof of part (3).
	Since (2) holds for $\mathcal{F}_1$ and $\mathcal{F}_{-1}$, and since the normal Euler number is additive under connected sum, part (2) follows, as well.
\end{proof}

Recall that a shadow diagram is doubly-standard if two of the pairings of arcs yield embedded curves.  We can use Proposition~\ref{spancap} to obtain the following classification result for doubly-standard diagrams.

\begin{theorem}\label{cor:doublestandard}
If $\Ss$ has a doubly-standard shadow diagram, then $\Ss$ is unknotted.
\end{theorem}

Note that Theorem~\ref{cor:doublestandard} also applies to surfaces with triply-standard shadow diagrams, as a special class of doubly-standard shadow diagrams.

\begin{proof}
Suppose $\Ss$ has a shadow diagram $(A,B,C)$ such that the pairings $(A,B)$ and $(B,C)$ are standard.  Consider the standard Heegaard splitting $\partial X_3= S^3 = H_+ \cup_{\Sigma} H_-$, and let $\Sigma_{\pm}$ be a parallel copy of $\Sigma$ pushed slightly into $H_{\pm}$.  Note that $A \cup B$ may have nested components (so that components of $A \cup B$ don't necessarily bound a collection of disjoint disks).  After a sequence of arc slides, however, performed only on the arcs in $A$, we obtain arcs $A'$ such that the embedded curves $A' \cup B$ bound a pairwise disjoint collection of disks.  We perform a similar procedure with $B \cup C$ to obtain $B \cup C'$. Now, embed parallel copies $A'_+ \cup B_+$ of the curves $A' \cup B$ in $\Sigma^+$ so that they bound a pairwise disjoint collection $D_+$ of disks in $\Sigma_+$, and embed parallel copies $B_- \cup C'_-$ of the curves $B \cup C'$ in $\Sigma_-$ so that they bound a pairwise disjoint collection $D_-$ of disks in $\Sigma_-$.  In $H_+$, there is an isotopy of $B_+$ to $B \subset \Sigma$ taking the disks $D_+$ to disks $D_1 \subset H_+$ such that $D_1 \cap \Sigma = B$.  The tangle $\T_1=\Ss\cap(H_+)$ is the image of $A'_+$ under this isotopy.  Similarly, in $H_-$ there is an isotopy of $B_-$ to $B$ taking the disks $D_+$ to disks $D_2 \subset H_-$ such that $D_2 \cap \Sigma = B$.  The tangle $\T_3=\Ss\cap H_-$ is the image of $C'_-$ under this isotopy. See Figure \ref{fig:doublystandard}.

By construction $D_1 \cap D_2 = B$, so that $F = D_1 \cup D_2$ is a spanning surface for the unlink $\T_1 \cup \T_3$.  Note further that $D_1$ is a trivial disk system for $\T_1 \cup B$, and $D_2$ is a trivial disk system for $B \cup \T_3$; hence, $\Ss$ is the union of $D_1, D_2$, and $D_3$, where $D_3$ is a trivial disk system for $\T_1 \cup \T_3$ pushed into $B^4$.  However, since $F  = D_1 \cup D_2 \subset S^3$, it follows that $\Ss$ is also isotopic to the cap-off $\Ff$ of $F$, which is unknotted by Proposition~\ref{spancap}.
\end{proof}

We are now ready to prove our main result.

\begin{theorem}\label{proc}
	If $\Ss$ is connected and $e(\Ss) = 0$, then there is a procedure to produce a Seifert solid for $\Ss$ that takes as input a tri-plane diagram for $\Ss$.
\end{theorem}

\begin{proof}
	The proof follows from the proofs of Propositions~\ref{orspan} and~\ref{nonorspan} below.
\end{proof}

In Section~\ref{subsec:getheegaard}, we show that there is a procedure to produce a Heegaard splitting for the Seifert solid when $\Ss$ is a 2--knot.

In addition to providing the proof of the above theorem, the next two propositions provide alternate proofs of the results in~\cite{Gluck62} and~\cite{gordonlitherland} mentioned above.

\begin{figure}{\centering
\includegraphics[width=.9\textwidth]{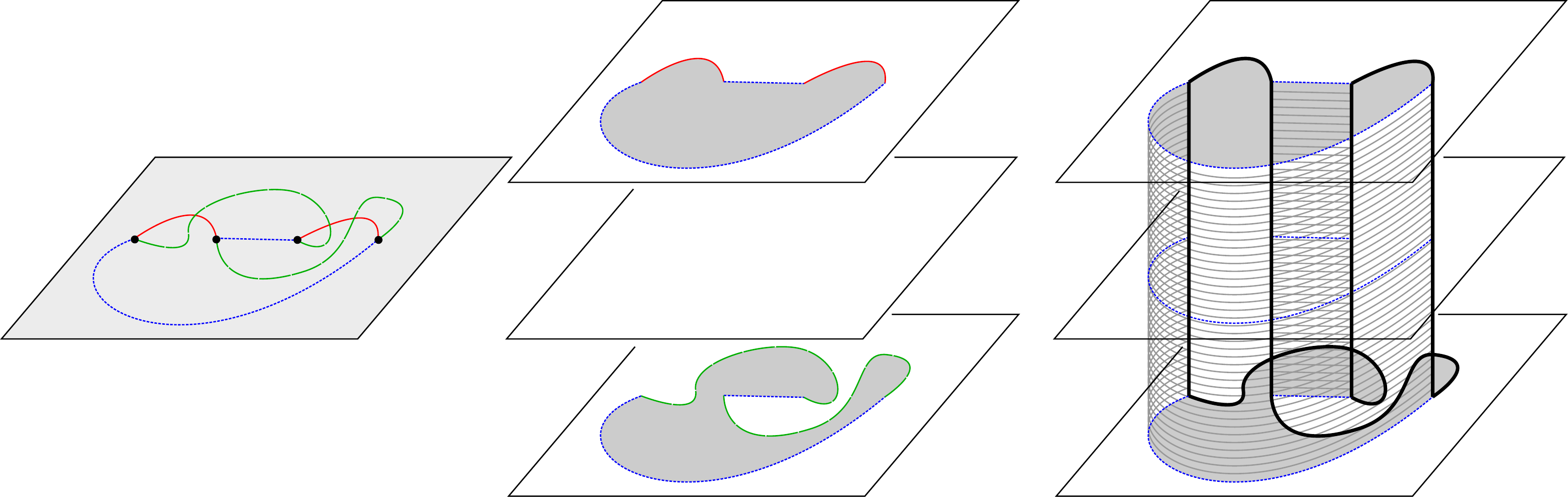}
\caption{Left: a doubly standard shadow diagram $(A,B,C)$. The pairings $(A,B)$ and $(B,C)$ are standard. Middle: disks in $\Sigma_+$ and $\Sigma_-$ bounded by parallel copies of $A\cup B$ and $B\cup C$, respectively. Right: A spanning surface $F$ for $\T_1\cup\T_2$ in $\partial X_3=S^3$.}
\label{fig:doublystandard}
}
\end{figure}

\begin{proposition}\label{orspan}
Every orientable surface-link $\Ss$ bounds a Seifert solid in $S^4$.
\end{proposition}

\begin{proof}
Let $\DD$ be a tri-plane diagram for $\Ss$, with induced orientation on the bridge points $\mathbf{x}$.  Perform mutual braid transpositions so that the bridge points alternate sign (orientation).  Then there are $b$ pairwise disjoint arcs $\eps$ contained in the equator $e$ connecting bridge points of opposite signs, so that $\DD_i \cup \eps$ is an oriented link diagram.  Let $F_i$ be the Seifert surface obtained by performing Seifert's procedure on the diagram $\DD_i \cup \eps$, and let $\widehat F_i = F_i \cup \overline F_{i+1}$ be the spanning surface obtained by gluing $F_i$ to $\overline F_{i+1}$ along $\eps$.  By Proposition~\ref{spancap}, there exists a handlebody $V_i \subset X_i$ such that $\pd V_i = \widehat F_i \cap \D_i$ and $V_i \cap \pd X_i = \widehat F_i$.  Finally, $Y = V_1 \cup V_2 \cup V_3$ is an embedded 3--manifold whose boundary is $\D_1 \cup \D_2 \cup D_3 = S$, and so $Y$ is a Seifert solid for $\Ss$.
\end{proof}

\begin{proposition}\label{nonorspan}
If $\Ss$ is connected and $e(\Ss) = 0$, then $\Ss$ bounds a spanning solid in $S^4$.
\end{proposition}

\begin{proof}
Consider a bridge trisection $\TT$ of $\Ss$, with $U_i = \pd \D_i$ and $\tau = \T_1 \cup \T_2 \cup \T_3$.  By taking, for example, a tri-plane diagram $\DD$ and compatible checkerboard surfaces in $\DD_i$, we can produce spanning surfaces $\widehat F_i$ for $U_i$ such that $\widehat F_i \cap H_i = \widehat F_{i-1} \cap H_i$.  Let $F_i$ denote $\widehat F_i \cap H_i$.  For each component $J$ of $U_i = \pd \widehat F_i$, let $\iota_{\widehat F}(J)$ denote the induced boundary slope on the curve $J$ by the surface $\widehat F_i$.  Then by Proposition~\ref{spancap}, we have
$$\sum_{J \subset U_1\cup U_2\cup U_3} \iota_{\widehat F}(J) = 0.$$

Choose a triple of spanning surfaces $\widehat F_i$ such that $\sum |\iota_{\widehat F}(J)|$ is minimal over all possible choices.  We claim that $\sum |\iota_{\widehat F}(J)| = 0$.  If not, then there exist boundary curves $J_+$ and $J_-$ such that $\iota_F(J_+) > 0$ and $\iota_F(J_-) < 0$.  Noting that the surface $\Ss$ contains all curves $J \subset U_i \subset \tau$, push each curve $J\subset U_i$ slightly off of $\tau$ into the corresponding disk component of $\D_i$, so that the collection of curves $J$ is embedded in $\Ss$ and disjoint from $\tau$.  Choose a path $\gamma \subset \Ss$ from $J_+$ to $J_-$, avoiding the bridge points, noting that $|\gamma \cap \tau| > 0$.  At each point of $\gamma \cap \tau$, modify the the corresponding component of $F_i$ by taking the boundary connected sum of $F_i$ with a trivial M\"obius band to obtain new surfaces $\widehat F_i'$ and $F_i'$, so that the corresponding boundary curves satisfy $\iota_{F'}(J'_+) = \iota_F(J_+) - 2$, $\iota_{F'}(J'_-) = \iota_F(J_-) + 2$, and $\iota_{F'}(J') = \iota_{F}(J)$ for all other curves $J'$.  It follows that $\sum |\iota_{F'}(J')| < \sum|\iota_F(J)|$, contradicting our assumption of minimality. (Note that $\iota_F(J)$ is always even, since it represents the number of intersection points between the boundary curves of spanning surfaces; see the proof of Lemma~\ref{compress}.)

We conclude that $\iota_F(J) = 0$ for all curves $J$, and thus by Proposition~\ref{spancap}, each spanning surface $\widehat F_i$ cobounds a (possibly) nonorientable handlebody $V_i \subset X_i$ with the disks $\D_i$.  It follows that $V_1 \cup V_2 \cup V_3$ is a spanning solid for $\Ss$ in $S^4$.
\end{proof}

\subsection{Procedure to find a Heegaard diagram for a Seifert solid}\label{subsec:getheegaard}

In this subsection, we describe a procedure for finding a Heegaard diagram for the Seifert solid coming from a bridge trisection $\TT$ of a 2--knot $\Ss$.  We use labels consistent with those appearing above in the proof of Proposition~\ref{orspan}. The process is illustrated in Figures~\ref{ex1a} through~\ref{ex2b}.

\textbf{Step 1}:  Given a tri-plane diagram $\DD$ for $\Ss$ perform interior Reidemeister moves and mutual braid transpositions so that the induced Seifert surfaces satisfy the following conditions:
\be
\item[(a)] Each of $F_{1}$, $F_{2}$, and $\widehat F_1$ is a collection of disks.
\item[(b)] Surfaces $\widehat F_2$ and $\widehat F_3$ are connected.
\item[(c)] $g(\widehat F_2) = g(F_{3})$.
\ee
See Figure \ref{ex1a}. Note that attaining condition (a) is possible since any tri-plane diagram can be converted to one in which two of the tangles have no crossings.  Condition (b) can be attained by performing interior Reidemeister moves on the diagram $\DD_{3}$.  Attaining condition (c) is possible since we can arrange so that $F_{2}$ is a collection of $b$ bridge disks, in which case $\widehat F_2$ deformation retracts onto $F_{3}$ (although in general, we need not assume that $F_{2}$ has $b$ components, as shown below).

\begin{figure}
  \centering
\labellist
\pinlabel {$F_{1}$} at 130 -50
\pinlabel {$F_{2}$} at 520 -50
\pinlabel {$F_{3}$} at 950 -50
\endlabellist
  \includegraphics[width=.5\linewidth]{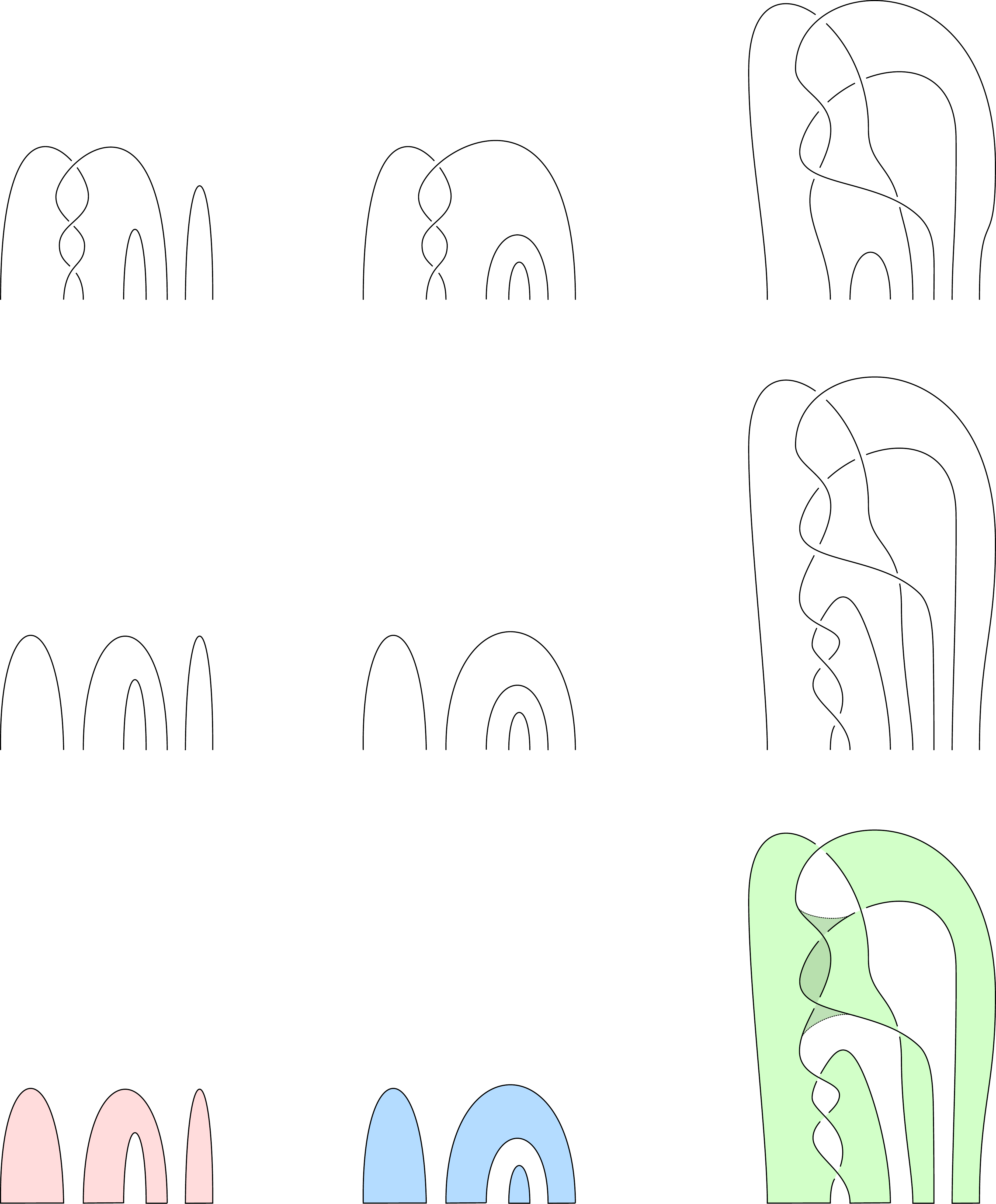}
\vspace{.1in}
  \caption{To perform the Seifert solid procedure on a tri-plane diagram, we first perform mutual braid transposition until the tangle diagrams in $V_{1}$ and $V_{2}$ have no crossings. Then we perform the usual Seifert's procedure for knot diagrams to obtain surfaces $F_{1}$, $F_{2}$, $F_{3}$ that agree in the bridge sphere $\Sigma$, with $F_{1}$ and $F_{2}$ and $\widehat F_1$ all collections of disks and $g(\widehat F_2)=g(F_{3})$.}
\label{ex1a}
\end{figure}

\textbf{Step 2}:  Following the proof of Proposition~\ref{spancap}, the surfaces $\widehat F_2$ and $\widehat F_3$ compress completely to disks in $S^3$.  Let $\A$ be a complete collection of pairwise disjoint compressing curves in $\widehat F_3$, and let $\n$ be a complete collection of pairwise disjoint compressing curves in $\widehat F_2$. See Figure~\ref{ex1b} (top row).

\textbf{Step 3}:  If necessary, slide the curves $\n$ over the components of $\pd \D_2$ to obtain a collection of curves $\n' \subset F_{3}$.  Note that since $g(F_{3}) = g(\widehat F_2)$, as curves in $\Ff_2 = \widehat F_2 \cup \D_2$, the collection $\n$ can be isotoped to be contained in $F_{3}$, and any isotopy of a curve over a disk component of $\D_2$ can be realized as a slide over $\pd \D_2$.  Thus, such a sequence of slides exists. See Figure \ref{ex1b} (middle row).

\textbf{Step 4}:  Let $P = \D_1 \cup \D_2$, so that $P$ is a planar surface with $c_3$ boundary components, let $Q$ be the surface obtained by gluing $P$ to $\widehat F_3$ along their boundaries, and let $\A^*$ be a choice of $c_3 - 1$ boundary components of $P$ and some minimal number of curves in $\A$ so that $\A^*$ forms a cut system for $Q$.

\textbf{Step 5}:  Let $\n^*$ be the union of $\n'$ and a collection of curves in $Q$ obtained by the following instructions:  For each component of $J$ of $\pd D_1$, suppose that $J$ meets $d$ disk components of $F_{2}$.  Choose $d-1$ of these components, isotope them off of $F_{2}$ in $\Ff_2 = F_{2} \cup F_{3} \cup \D_2$, and add these $d-1$ curves to $\n^*$.  Discard any superfluous curves of $\n'$ so that $\n^*$ is a cut system for $Q$.

\begin{figure}
  \centering
\labellist
\pinlabel {$\alpha$} at -15 550
\pinlabel {$\beta$} at 340 550
\pinlabel {$\alpha'$} at -15 320
\pinlabel {$\beta'$} at 340 320
\pinlabel {$\cong$} at 160 90
\endlabellist
  \includegraphics[width=.5\linewidth]{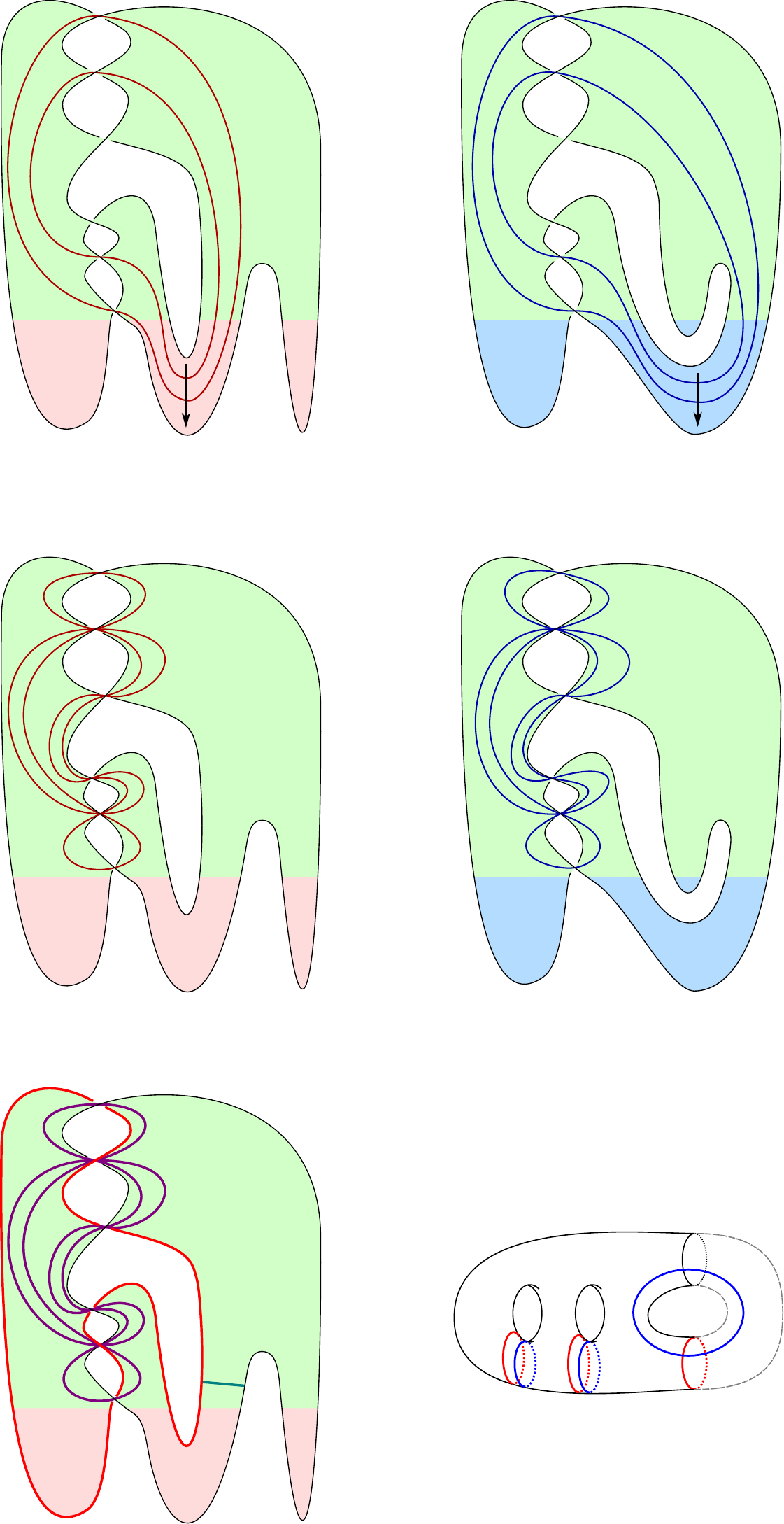}
  \caption{Top: we find complete sets of compressing curves $\alpha, \beta$ for $\widehat F_3$ and $\widehat F_2$, respectively. Middle: We slide $\alpha,\beta$ (with slides indicated in top row) over $\partial \widehat F_3,\partial \widehat F_2$ to obtain curve systems $\alpha',\beta'$ that are each completely within $F_{3}$. Bottom: We obtain $\alpha^*$ (red and purple curves) by adding boundary curves as in Step (5) of \S\ref{subsec:getheegaard}. We obtain $\beta^*$ by adding arcs as in Step (4). Then $(Q;\alpha^*,\beta^*)$ is a Heegaard diagram for a (closure of a) Seifert solid for the 2--knot described by the initial tri-plane diagram. }
\label{ex1b}
\end{figure}
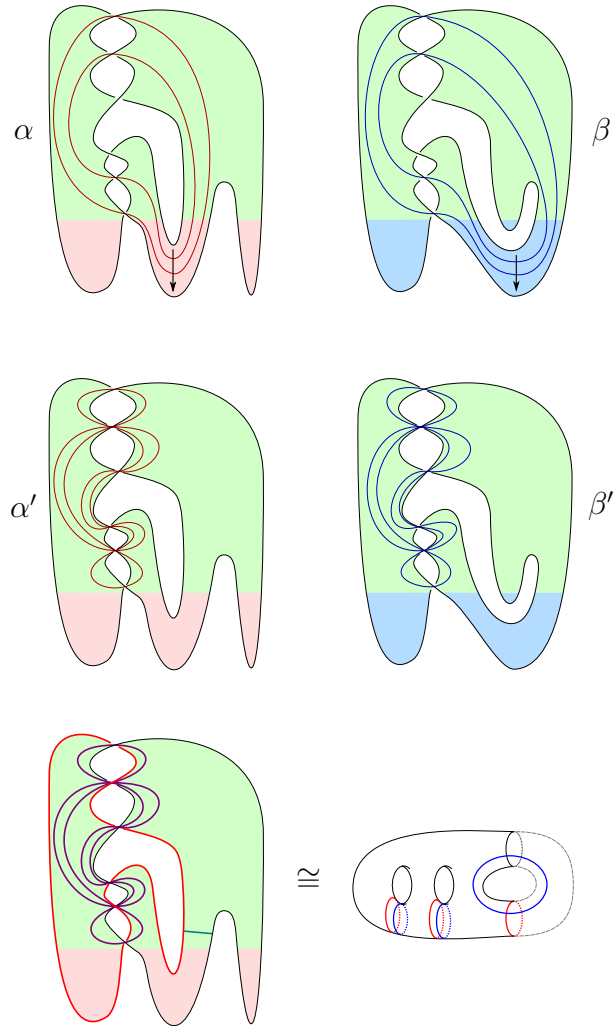

\begin{proposition}
Using the procedure described above, $\Ss$ bounds a punctured copy of the 3--manifold determined by the Heegaard diagram $(Q;\A^*,\n^*)$.  
\end{proposition}

\begin{figure}
  \centering
  \centering
\labellist
\pinlabel {$\widehat F_3$} at 85 135
\pinlabel {$\widehat F_2$} at 195 135
\endlabellist
  \includegraphics[width=3in]{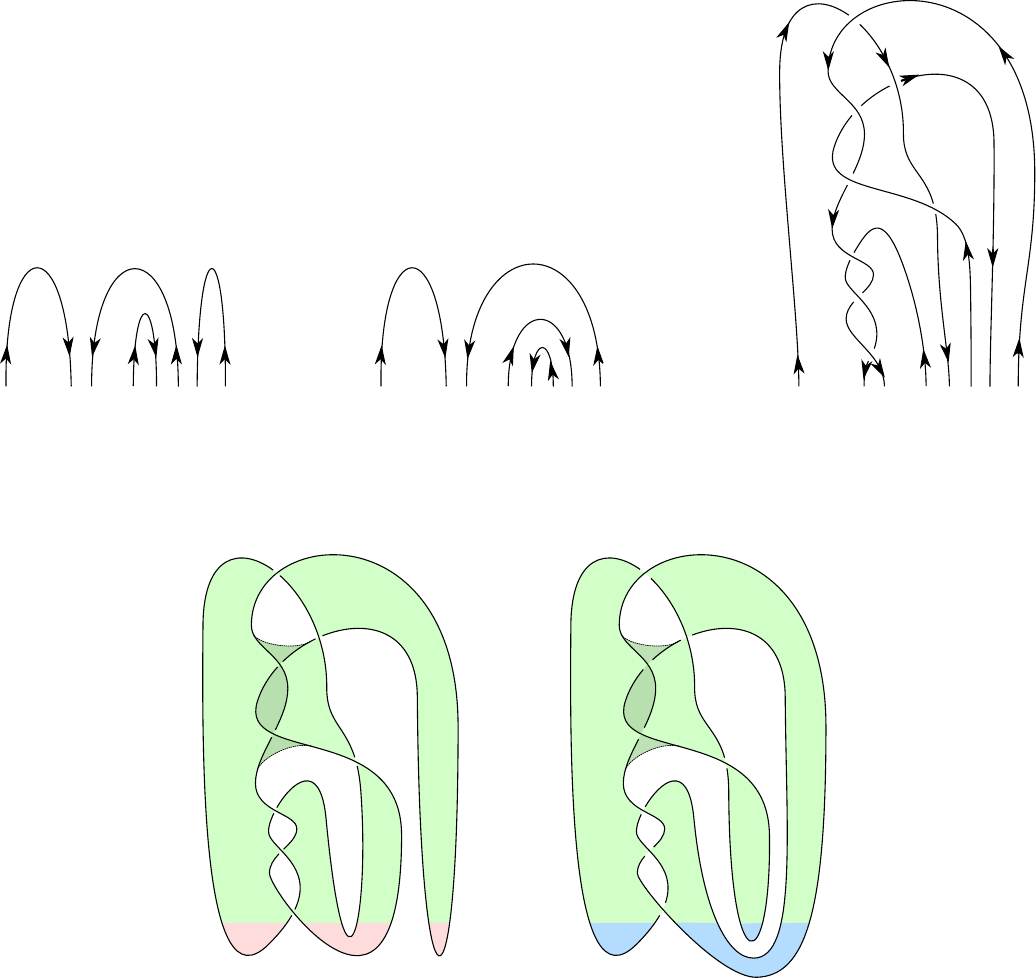}
  \caption{We start performing the Seifert solid procedure (\S\ref{subsec:getheegaard}) on the tri-plane diagram in the top row.}
\label{ex2a}
\end{figure}

\begin{proof}
Suppose that $\DD$ is a tri-plane diagram satisfying conditions (a), (b), and (c) given in Step 1 above.  Following the proofs of Proposition~\ref{spancap} and Proposition~\ref{orspan}, we have that for each $i$, the surface $\widehat F_i \cup \D_i$ bounds a handlebody $V_i$, where $V_1$ is a collection of 3--balls, say $B_1,\dots,B_n$, and $V_2$ and $V_3$ are connected.  Moreover, $\A$ contains a cut system for $V_3$ and $\n$ contains a cut system for $V_2$.  Since $\n'$ is homotopic to $\n$ in $\pd V_2$, it follows that $\n'$ also contains a cut system for $V_2$.  Thus, the Seifert solid bounded by $\Ss$ is equal to $V_2 \cup V_3 \cup B_1 \cup \dots \cup B_n$.  Let $Y$ be the closed 3--manifold obtained by capping off the boundary $\Ss$ of this Seifert solid with an abstract 3--ball $B_0$.  We will show that $(Q;\A^*,\n^*)$ is a Heegaard diagram for $Y$.

To this end, consider $W = V_3 \cup B_0$ and $W' = V_2 \cup B_1 \cup \dots \cup B_n$.  Considering that $\pd V_2 = F_{2} \cup F_{3} \cup \D_2$ and $\pd (B_1 \cup \dots \cup B_n) = F_{1} \cup F_{2} \cup \D_3$, we have that
\[ \pd W' = F_{3} \cup F_{1} \cup \D_2 \cup \D_3 = \widehat F_3 \cup P = Q.\]
Additionally, the 3--balls $B_i$ are attached to $V_2$ along $F_{2}$, which is a collection of disks by condition (a).  It follows that the curves $\n' \cup \pd F_{2}$ bound compressing disks in $W'$ cutting $W'$ into a collection of 3--balls, so $W'$ is a handlebody.  In addition, choosing all but one curve of $\pd F_{2}$ for each component $B_i$ and a subset of $\n'$ as in Step 5 above yields a cut system $\n^*$ for $W'$.

Turning our attention to $W$, we have $\pd V_3 = \widehat F_3 \cup \D_3$ and $\pd B_0 = \D_1 \cup \D_2 \cup \D_3$, so that $\pd W = \widehat F_3 \cup \D_1 \cup \D_2 = Q$, and in addition, the curves $\A$ and $\pd \D_3$ bound disks cutting $W$ into 3--balls.  Choosing $\A^*$ to contain all but one curve of $\pd \D_3$ and a subset of $\A$ as in Step 4, we have that the curves in $\A^*$ bound disks cutting $W$ into a single 3--ball, so $\A^*$ is a cut system for $W$.  We conclude that $(Q;\A^*,\n^*)$ is a Heegaard diagram for $Y$, as desired.
\end{proof}

\begin{remark}
It may be the case that the surface $F_{3}$ compresses in $H_{3}$, in which case $\A$ and $\n$ could have one or more curves in $F_{3}$ in common.  Following the procedure with such $\A$ and $\n$ produces one or more extra $S^1 \X S^2$ summands for the 3--manifold $Y$, and a simpler Seifert solid can be obtained by first compressing $F_{3}$ maximally in $H_{3}$.
\end{remark}

\begin{remark}
The procedure above can be generalized:  We can relax conditions (a), (b), and (c) from Step 1; the only assumption necessary to ensure that $V_1 \cup V_2$ is a handlebody is that their intersection $F_{2}$ is a collection of disks.  However, the weaker conditions make it somewhat more difficult to draw the diagram, since we are no longer guaranteed the existence of the slides of Step 3 -- it may be the case that $\n$ curves necessarily intersect the disks $\D_1$ and $\D_2$.
\end{remark}

\begin{remark}
The observant reader might notice that we call our process the Seifert solid {\emph{procedure}}, rather than {\emph{algorithm}}. An algorithm gives an output completely determined from the input, independent of further choices. A procedure may require additional choices for the output to be determined.  In the procedure we give in this section to find a description of a Seifert solid for a 2--knot, we are forced to choose compressing circles for surfaces in $S^3$. These circles are generally not unique (and in fact, different choices can determine different Seifert solids), so we do not refer to this procedure as an algorithm.
\end{remark}

\subsection{Some examples}

In this subsection, we carry out the procedure described above for a couple of specific examples.  The first is the spun trefoil.  In Figure~\ref{ex1a}, we see a tri-plane diagram for the spun trefoil coming from~\cite{MeiZup_bridge1}, followed by the result of performing tri-plane moves so that the induced Seifert surfaces $F_{i}$ satisfy conditions (a), (b), and (c) from Step 1 above.

\begin{figure}
  \centering
\labellist
\pinlabel {$\alpha'$} at 70 125
\pinlabel {$\beta'$} at 210 124
\pinlabel {$\cong$} at 402 66
\endlabellist
  \includegraphics[width=.9\textwidth]{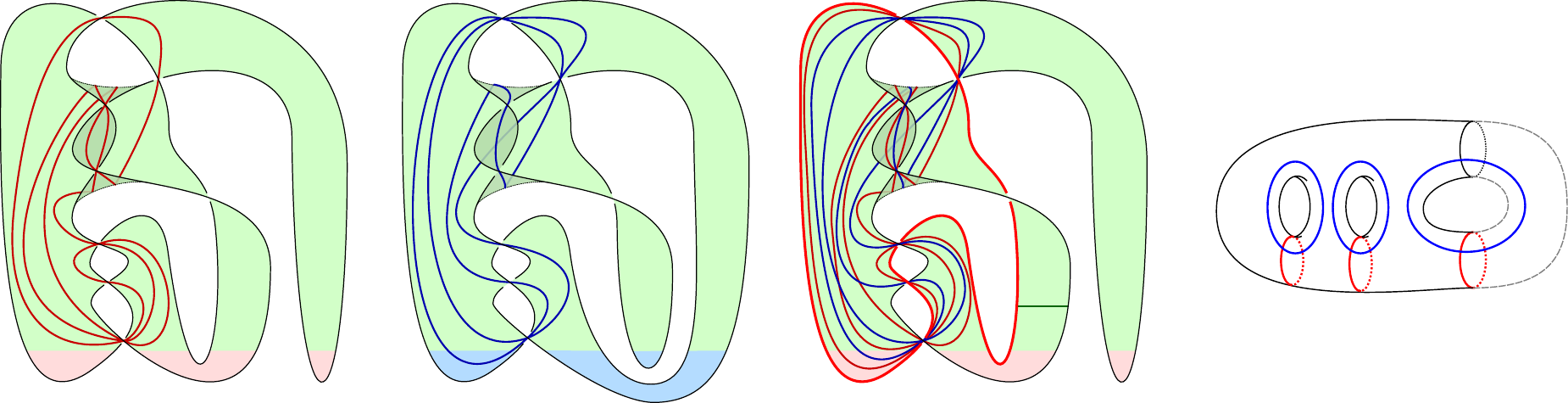} 
  \caption{Leftmost: The curves $\alpha'$ in $F_3$. Second: The curves $\beta'$ in $F_2$. Third: we add some boundary curves of $F_2$ to $\alpha$ to obtain $\alpha^*$ and some arc to $\beta'$ to obtain $\beta^*$. Rightmost: we simplify the resulting Heegaard diagram $(\Sigma;\alpha^*,\beta^*)$ to see that it is a diagram of $S^3$. Thus, the initial 2--knot bounds a copy of $B^3$ in $S^4$, so is unknotted.}
\label{ex2b}
\end{figure}

In the top panel of Figure~\ref{ex1b}, we find the compressing curves $\A$ on $\widehat F_3$ and $\n$ on $\widehat F_2$.  Note that in this case $\D_3$ contains two disks, so that $P = \D_1 \cup \D_2$ is an annulus, and $Q = \widehat F_3 \cup P$ can be obtained by identifying the two boundary components of $\widehat F_3$.  Under this identification, the identified boundary components constitute the third curve in the cut system $\A^*$.  In the second panel at left, we slide the two curves of $\A$ over the third curve of $\A^*$ in $Q$.  In the second panel at right, we slide the two curves of $\n$ over a boundary component as shown to get the curves $\n' \subset F_{3}$ (which are identical to the image of $\A$ under the slides described above).  Finally, the third curve of $\n^*$ consists of the teal arc depicted in $F_{3}$ and a spanning arc in the annulus $A$, or equivalently, we can identify the endpoints of the teal arc.  In the lower panel, we see the diagram for the Seifert solid, the standard (once-stabilized) Heegaard diagram for $\#^2(S^1 \X S^2)$.

\begin{remark}
These diagrams and arguments easily generalize to produce the Seifert solid $\#^{p-1}(S^1 \X S^2)$ for the spun $(p,2)$-torus knot. Miyazaki proved that the degree of the Alexander polynomial (over $\mathbb{Q}[t,t^{-1}]$) is a lower bound for the second Betti number of any Seifert solid \cite{miyazaki}. Since the degree of the Alexander polynomial of $T(2,p)$ is $p-1$, these solids are minimal in the sense that the corresponding 2--knots cannot bound any 3-manifold with a smaller second Betti number, e.g.\ a fewer number of $S^1 \X S^2$ summands.
\end{remark}

For the second example, we find a Seifert solid for the 1-twist spun trefoil (which is unknotted by~\cite{Zeeman}).  In Figure~\ref{ex2a}, we include a simplified tri-plane diagram for the 1-twist spun trefoil along with the surfaces $\widehat F_2$ and $\widehat F_3$ this diagram generates.

Next, we find the compressing curves $\A$ for $\widehat F_3$ and $\n$ for $\widehat F_2$.  As in the spun trefoil example above, $P = \D_1 \cup \D_2$ is an annulus, so we view $Q$ as being obtained by identifying the two boundary components of $\widehat F_3$, with this identified boundary the third curve in $\A^*$.  Figure~\ref{ex2b} shows the curves $\A$, $\n$, and the union of the sets in $Q$, yielding the standard diagram for $S^3$, in which the third curve of $\n^*$ appears as a teal arc with boundary points identified (as above).  Note that the existence of the curves $\A$ and $\n$ is guaranteed by Proposition~\ref{spancap}; in practice, however, these curves are found using ad hoc methods.  


\section{Spinal Seifert solids}
\label{sec:afss}

A natural aspect of the study of Seifert surfaces for links in the 3--sphere is the consideration their exterior.
We call a Seifert surface $F$ for $L$ \emph{canonical} if it is isotopic to a surface obtained by applying Seifert's procedure to a diagram for $L$.
We call a Seifert surface $F$ \emph{free} if its exterior $S^3\setminus\nu(F)$ is a 3--dimensional handlebody -- equivalently, has free fundamental group.
It is an easy exercise to see that a canonical Seifert surface is free, provided that it is connected; so every link admits a free Seifert surface, by the application of Seifert's algorithm to a non-split diagram.
However, such a surface can be far from minimal genus.
M.~Kobayashi and T.~Kobayashi showed that the difference between the genus of a knot and the minimal genus of a free Seifert surface for the knot can be arbitrarily large, and that moreover the difference between the minimal genus of a free Seifert surface for a knot and the minimal genus of a canonical Seifert surface can also be arbitrarily large~\cite{Kob-Kob}. (In fact, they show that both of these differences can be made arbitrarily large at the same time.)

In this section, we introduce 4--dimensional analogues of the notions of canonical and free Seifert surfaces.
Going forward, let $\Ss\subset S^4$ be a surface-link admitting a Seifert solid.
(This is equivalent to the condition that $\Ss$ be orientable or have normal Euler number zero.)
We call a Seifert solid $Y$ \emph{canonical} if it is isotopic to a Seifert solid obtained by the procedure given in Section~\ref{subsec:existence} (see Propositions~\ref{orspan} and~\ref{nonorspan}). 
We call a Seifert  solid $Y$ \emph{spinal} if  $S^4\setminus\nu(Y)$ deformation retracts onto a finite 2--complex.
Equivalently, $S^4\setminus\nu(Y)$ can be built with handles of index at most two.

\begin{theorem}
\label{thm:free}
	If a surface-knot $\Ss$ admits a Seifert solid, then it admits a canonical Seifert solid that is spinal.
\end{theorem}

\begin{proof}
	First, note that in the proof of Propositions~\ref{orspan} and \ref{nonorspan}, it is possible to arrange that each Seifert surface $F_i$ is connected:
	For example, this is assured if each $\DD_i\cup\varepsilon$ is non-split.
	Let $Y$ be a canonical Seifert solid for $\Ss$ given by Proposition~\ref{orspan} or Proposition~\ref{nonorspan} such that the canonical surface $F_i = Y\cap H_i$ is connected for each $i\in\Z_3$.
	We make use of the notation of the proof of Proposition~\ref{orspan} in what follows.
	
	Recall that $V_i = X_i\cap Y$ is a handlebody with $\partial V_i = \widehat F_i\cup\D_i$.
	Moreover, $V_i$ is built relative to $\widehat F_i$ by attaching 3--dimensional 2--handles and 3--handles.
	It follows that $X_i\setminus\nu(V_i)$ can be built with 4--dimensional 0--, 1--, and 2--handles.
	
	Next, recall that $F_i$ is a canonical Seifert surface for the link $\DD_i\cup\varepsilon$, considered in $S^3 = H_i\cup_\Sigma B^3$.
	Since we have assumed $F_i$ is connected, we have that $F_i$ is free in $H_i\cup_\Sigma B^3$.
	Since $\varepsilon\subset\partial H_i$, it follows that $H_i\setminus F_i$ is also a 3--dimensional handlebody.
	 
	Finally, we can build $S^4\setminus\nu(Y)$ by taking the $X_i\setminus\nu(V_i)$ and gluing them along the $H_i\setminus\nu(F_i)$.
	Since the three gluings occur along 3--dimensional handlebodies, it follows that $S^4\setminus\nu(Y)$ is obtained from the disjoint union of the $X_i\setminus\nu(V_i)$ by attaching 4--dimensional 1-- and 2--handles.
	Because each of the $X_i\setminus\nu(V_i)$ were built with 4--dimensional handles of index at most two, the same is true for $S^4\setminus\nu(Y)$.
	This shows that $Y$ is spinal, as desired.
\end{proof}

When studying Seifert surfaces, the genus of the surface is the obvious measure of complexity that one might try to minimize.
In contrast, there are many ways one might try to quantify the complexity of a Seifert solid $Y$ for a surface-knot; indeed, any complexity one might associate to a 3--manifold could be interesting to consider.
Here, we content ourselves to give some examples showing that there is at least one sense in which a simple Seifert solid for a surface-knot can be arbitrarily far from being spinal.

\begin{theorem}
\label{thm:not_free}
	Given any $n\in\N$, there exists a 2--knot $\K$ that bounds a Seifert solid $Y$ homeomorphic to $(S^1\times S^2)^\circ$ such that $S^4\setminus\nu(Y)$ requires at least $n$ 4--dimensional 3--handles.
\end{theorem}

\begin{proof}
	Let $J$ be an arbitrary knot, and let $K = \text{Wh}_0(J\#\overline J)$ be the untwisted Whitehead double of the connected sum of $J$ with its mirror.
	Let $F$ be the standard genus one Seifert surface for $K$, and let $\gamma$ be the curve on $F$ that is isotopic to $J\#\overline J$.
	(Alternatively, $F$ is obtained by taking a 0--framed annular thickening of a curve $\gamma$ isotopic to $J\#\overline J$ and plumbing on a Hopf band.)
	
	Let $E$ be the standard ribbon disk for $\gamma$, so that $(B^4, E) = (S^3,J)^\circ\times I$. The surface $F$ can be surgered along $E$ in the 4--ball to get a slice disk $D$ for $K$, and the trace of this surgery yields a solid torus $V$ with $\partial V = F\cup D$.
	
	Let $\K=D\cup_K\overline D$ be the 2--knot obtained by doubling $D$, and let $Y=H\cup_F\overline H$ be the double of $H$ along $F$.
	Then, $Y$ is a Seifert solid for $\K$ and $Y\cong(S^1\times S^2)^\circ$.
	
	We claim that $\pi_1(S^4\setminus\nu(Y))\cong\pi_1(S^3\setminus\nu(J))$.
	First, we have $\pi_1(S^4\setminus\nu(Y))\cong\pi_1(B^4\setminus H)$, since the former exterior is the double of the latter exterior along the exterior of $F$ in $S^3$ and $\pi_1(S^3\setminus\nu(F))$ surjects onto $\pi_1(B^4\setminus H)$ under inclusion.
	Next, by construction, $H$ is obtained by thickening the slice disk $E$ and attaching a trivial 3--dimensional 1--handle. It follows that
	$$\pi_1(B^4\setminus\nu(H))\cong\pi_1(B^4\setminus\nu(E))\cong\pi_1(S^3\setminus\nu(J)),$$
	as desired.
	
	To complete the proof, let $n\in\N$ be given, and choose $J$ to be any knot with $\text{rank}(\pi_1(S^3\setminus\nu(J)))\geq n+2$ (e.g.\ take $J$ to be a connected sum of $n+1$ trefoils \cite{Weidmann}). 
	The exterior $S^4\setminus\nu(Y)$ can be built relative to $\partial(S^4\setminus\nu(Y))\cong (S^1\times S^2)\#(S^1\times S^2)$ with some number of 4--dimensional 1--, 2--, 3--, and 4--handles.
	Since the 1--handles correspond to generators of the fundamental group, at least $n$ are required; the boundary $\partial(S^4\setminus\nu(Y))$ contributes only two to the rank of the fundamental group. Similarly, since we can obtain another presentation of $\pi_1(S^4\setminus\nu(Y))$ with generators corresponding to 3--handles, the number of 3--handles in this decomposition is at least $n+2$.
\end{proof}

We note that the construction of $\K$ given in the above proof is closely related to an interesting construction of 2--knots given by Cochran~\cite{cochran}.

Next, we observe that many important examples of Seifert solids are, in fact, spinal:
\begin{enumerate}
	\item Every ribbon 2--knot bounds a Seifert solid $Y$ that is homeomorphic to $(\#^m(S^1\times S^2))^\circ$ for some $m$ \cite{Yanagawa}. The manifold $Y$ is obtained by taking a Seifert surface $F$ for some ribbon knot in an equatorial $S^3$, thickening it, and attaching trivial 2--handles above and below the equator. By attaching tubes to $F$ (at the cost of increasing $m$), we can arrange for $F$ to be free. Then $Y$ is spinal.
	\item If $\K$ is fibered with fiber $Y$, then $S^4\setminus\nu(Y)\cong Y\times I$ is an spinal, since $Y$ is a punctured 3--manifold.
	\item Connected Seifert solids arising from broken surface diagrams via the construction given by Carter and Saito~\cite{SeifertAlgorithm} are spinal. Recall that a connected, canonical Seifert surface is free because it deformation retracts to a graph so that on each edge, there is one local maximum and no local minima with respect to the radial height function on $S^3$. (Here, the vertices of the graph correspond to the disks produced in Seifert's procedure while the edges correspond to the half-twisted bands.) This ensures that the exterior of a canonical surface can be built with 0-- and 1--handles. Similarly, a Seifert solid constructed \`a la \cite{SeifertAlgorithm} deformation retracts to a 2--complex with one local maximum and no other critical points in the interior of each 1-- and 2--cell. Thus, the exterior of such a Seifert solid can be built with 0--, 1--, and 2--handles.
\end{enumerate}

Finally, we can formulate a question analogous to the 3--dimensional results in~\cite{Kob-Kob} in the setting of surface-knots.

\begin{question}
	Define the \emph{genus} of an orientable surface-knot $\Ss$ in $S^4$ to be the minimal first Betti number of any Seifert solid bounded by $\Ss$, and define the \emph{spinal genus} and \emph{canonical genus} similarly, using spinal Seifert solids and canonical Seifert solids, respectively.  Do there exist surface-knots for which these three measures of complexity differ?
\end{question}

We remark that using techniques as in the proof of Theorem~\ref{thm:not_free}, one can show that for some of the known classical knots $K$ whose genus and free genus are sufficiently different (see~\cite{Moriah}, for example), the spun knots $\Ss(K)$ admit low-complexity non-spinal Seifert solids, whereas the obvious spinal and canonical Seifert solids have greater complexity.  However, it is likely to be considerably more difficult to obstruct the existence of low-complexity spinal or canonical Seifert solids, even for these examples.

\section{On standardness of bridge trisections}
\label{sec:c=b}

The goal of this section is to prove Theorem~\ref{thm:c=b}, which states that a $(b;c_1,c_2,c_3)$--bridge trisection that satisfies $c_i\geq b-1$ for some $i\in\Z_3$ can be completely decomposed into standard pieces.  This proves Conjecture~4.3 of~\cite{MeiZup_bridge1}, and the theorem can be viewed as the bridge trisection analog of the main result in~\cite{MSZ}, which states that every $(g;k_1,k_2,k_3)$--trisection with $k_i \geq g-1$ for some $i$ is standard in that it decomposes into genus one summands.

We encourage the reader to recall the notions of \emph{perturbation} and \emph{connected summation} for bridge trisections. The former was first introduced in Section~6 of~\cite{MeiZup_bridge1}, where it was referred to as stabilization, and the latter can be reviewed in Subsection~2.2 of~\cite{MeiZup_bridge1}.  See also~\cite[Section~3]{MTZ_graph} for a succinct description of these concepts.

We call a surface-link an \emph{unlink} if it is the split union of unknotted surface-knots, though we allow the topology of each component to vary.  For example, one might have a 2--component unlink that is the split union of an unknotted 2--sphere and an unknotted projective plane.  (See Subsection~2.2 of~\cite{MTZ_graph} and Subsection~\ref{subsec:unknotted} above for a brief discussion of unknotted surface-knots.)

Before proving Theorem~\ref{thm:c=b} in generality, we recall the case in which $c_i=b$ for some $i\in\Z_3$.  This was addressed as Proposition~4.1 of~\cite{MeiZup_bridge1}. A bridge trisection is called \emph{completely decomposable} if it is a disjoint union of perturbations of  one-bridge and two-bridge trisections.

\begin{proposition}{\cite[Proposition~4.1]{MeiZup_bridge1}}
\label{prop:c=b}
	Let $\TT$ be a $(b;c_1,c_2,c_3)$--bridge trisection with $c_i=b$ for some $i\in\Z_3$. Then, $\TT$ is completely decomposable, and the underlying surface-link is the unlink of $\min_i\{c_i\}$ 2--spheres.
\end{proposition}

Note that if $c_i=b$ for some $i\in\Z_3$, then $c_{i-1}=c_{i+1}$.  Similarly, in what follows we will see that if $c_i=b-1$ for some $i\in\Z_3$, then $|c_{i-1}-c_{i+1}|\leq 1$.  We now present and prove the main result of this section.

\begin{theorem}
\label{thm:c=b}
	Let $\TT$ be a $(b;c_1,c_2,c_3)$--bridge trisection with $c_i=b-1$ for some $i\in\Z_3$.  Then, $\TT$ is completely decomposable, and the underlying surface-link is either the unlink of $\min\{c_i\}$ 2--spheres or the unlink of $\min\{c_i\}$ 2--spheres and one projective plane, depending on whether $|c_{i-1}-c_{i+1}|=1$ or $c_{i-1}=c_{i+1}$.
\end{theorem}

The key ingredient in the proof of the theorem is a pair of results of Scharlemann  and Bleiler-Scharlemann about planar surfaces in 3--manifolds~\cite{Sch_4crit,BleSch_proj}.  We refer the reader to Section~1 of each of these papers, as we will adopt the notation of~\cite[Theorem~1.1]{Sch_4crit} and~\cite[Theorem~1.3]{BleSch_proj} in the proof below.

\begin{proof}[Proof of Theorem~\ref{thm:c=b}]
	We induct on the bridge number $b$ of the bridge trisection.
	When $b=1$ or $b=2$, there is an easy classification of $b$--bridge trisections~\cite[Subsection~4.3]{MeiZup_bridge1}, which we take as the base case.
	Assume the theorem holds when  the bridge number is less than $b$, and let $\TT$ be a $(b;c_1,c_2,c_3)$--bridge trisection.
	Assume without loss of generality that $c_3=b-1$.
	
	Suppose that $\T_1$, $\T_2$, and $\T_3$ are the three tangles comprising the spine of the bridge trisection.
	Every $b$--bridge splitting of a $c$--component unlink with $b>c$ is a perturbation of the standard $c$--bridge splitting of the $c$--component unlink, which is itself unique up to isotopy~\cite[Proposition~2.3]{MeiZup_bridge1}.
	It follows that there exist collections $\Delta_1$ and $\Delta_3$ of bridge disks for $\T_1$ and $\T_3$, respectively, so that the shadows $\Delta_1^* = \Delta_1\cap\Sigma$ and $\Delta_3^* = \Delta_3\cap\Sigma$ have the property that $\Delta_1^*\cup\Delta_3^*$ is an embedded collection of $b-2$ bigons and a single quadrilateral.
	Let $\alpha_0^*$ denote one of the arcs of $\Delta_1^*$ in the quadrilateral.

	Let $L = \T_2\cup\overline\T_3$, and let $\mathfrak b$ be the band for $L$ that is framed by $\Sigma$ and whose core is $\alpha_0^*$.
	Then the data $(\Sigma,L,\frak b)$ encodes a banded $b$--bridge splitting, since the resolution $L_\frak b$ is the unlink $L' = \T_2\cup\overline\T_1$.
	(Here, we think of $\frak b$ as being slightly perturbed to lie in the 3--ball containing $\T_3$.)
	We refer the reader to Section~3 of~\cite{MeiZup_bridge1}, especially Lemma~3.3, for more details about banded bridge splittings and how they arise from bridge trisections.
	
	Assume without loss of generality that $c_2 = |L|$ is greater than or equal to $c_1 = |L'|$.
	We break the remainder of the proof into two cases:  Either $c_2>c_1$ or $c_2=c_1$.
	Note that since there is only one band present, we must have $c_2-c_1\leq 1$.
	The proofs of the two cases are very similar, except that we apply~\cite[Theorem~1.1]{Sch_4crit} in the first case and~\cite[Theorem~1.3]{BleSch_proj} in the second.
	
	\textbf{Case 1.} If $c_2 = c_1 + 1$, then $\frak b$ connects distinct components $K_1$ and $K_2$ of $L$.
	Let $K'$ denote the component of $L'$ obtained as the resolution $(K_1\cup K_2)_\frak b$.
	We now translate this set-up into the notation of~\cite[Section~1]{Sch_4crit}.
	Let $N=\overline{\nu(K_1\cup\frak b\cup K_2)}$, a genus two handlebody, and let $M = S^3\setminus\nu(L\setminus(K_1\sqcup K_2))$.
	Let $E_1$ denote the spanning disk bounded by $K_1$.
	Let $P' = \partial\nu(E_1)$, a 2--sphere disjoint from $K_1\sqcup K_2$ in $M$.
	Let $Q'$ denote a spanning disk bounded by $K'$ in $M$.
	Let $P = \overline{P'\setminus N}$, and let $Q = \overline{Q'\setminus N}$.

	It is clear from this set-up that $P\cap\partial N$ is a collection of $m$ parallel separating curves $A_m$ for some odd $m$, since $P'$ was disjoint from $K_1$ and $K_2$, but intersects $\frak b$ transversely.
	(See~\cite[Fig.~1]{Sch_4crit}.)
	Similarly, we have $Q\cap\partial N$ agrees with the curves $B_n$, since $\partial Q'=K'$ and $Q'$ may crash through $\frak b$ in arcs parallel to its core.
	Thus, $M$, $N$, $P$, and $Q$ satisfy the hypotheses of~\cite[Theorem~1.1]{Sch_4crit}.
	The relevant conclusion is that $A_1$ and $B_0$ bound embedded disks $E$ and $F$ in $\overline{M\setminus N}$ that intersect in a single arc.
	(Compare with the proof of~\cite[Main Theorem]{Sch_4crit}.)
	
	Translating this conclusion back into the setting of interest, we find that the disk $E$ is properly embedded in $S^3\setminus\nu(\frak b)$ and that $F$ is a spanning disk for $K'$.
	This implies that the pair $(B^3,T) = (S^3, L)\setminus(\nu(\frak b),\nu(L\cap\frak b))$ is the split union of a trivial tangle and an unlink: The strands of the trivial tangle are parallel into push-offs of $E$ via the components of $F\setminus\nu(E)$, at which point they are parallel into $\partial\nu(\frak b)$ via the push-offs of $E$.
	
	The bridge sphere $\Sigma$ induces a bridge splitting $(B^3,T)$.
	By Theorem~2.2 of~\cite{Zup_13_Bridge-and-pants}, $\Sigma$ is either minimal for $(B^3,T)$ or perturbed\footnote{Although Theorem~2.2 of~\cite{Zup_13_Bridge-and-pants}, as stated, applies to a closed 3--manifold $M$ and a link $K$ in $M$, a verbatim proof establishes the more general case where the 3--manifold $M$ is replaced by a punctured 3--manifold and the link $K$ is a tangle.}.
	If the splitting were minimal, we would have $b=c_2$, so $\TT$ would be completely decomposable by Proposition~\ref{prop:c=b}.
	If the splitting is perturbed, then $\TT$ is perturbed, since each bridge arc of $\T_3$ that is disjoint from $\nu(\frak b)$ is a strand of a 1--bridge splitting of a component of $L_3 = \T_3\cup\overline\T_1$.
	After de-perturbing $\TT$, we find that $\TT$ is completely decomposable, by the inductive hypothesis.
	
	\textbf{Case 2.} If $c_2 = c_1$, then $\frak b$ connects a component $K$ of $L$ to itself.
	Let $K' = K_\frak b$.
	We now translate this set-up into the notation of~\cite[Section~1]{BleSch_proj}, abbreviating the discourse where it is overly repetitive of the previous case.
	Let $M = S^3\setminus\nu(L\setminus K)$, and let $N = \nu(K\cup\frak b)$.
	Let $P'$ be a spanning disk bounded by $K$ in $M$, and let $Q'$ be a spanning disk bounded by $K'$ in $M$. Let $P = \overline{P'\setminus N}$, and let $Q = \overline{Q'\setminus N}$.
	
	It is clear from the set-up that the hypotheses of~\cite[Theorem~1.3]{BleSch_proj} are satisfied, so we can conclude that some $A_0$ and $B_0$ bound embedded disks $E_P$ and $E_Q$, respectively, in $\overline{M\setminus N}$.
	Moreover, there is a properly-embedded disk $D$ in $\overline{M\setminus N}$, disjoint from $E_P$ and $E_Q$, that runs once over one of the handles of $N$ and is disjoint from the other handle.
	We can extend $E_P$ to a spanning disk $F$ for $K$.
	(Compare with the proof of~\cite[Theorem~1.8]{BleSch_proj}.)
	
	The strands of $K\setminus\nu(\frak b)$ are parallel into push-offs of $D$ via the components of $E_P\setminus\nu(D)$, at which point they are parallel into $\partial\nu(\frak b)$ via the push-offs of $D$.
	It follows that the tangle $(B^3,T) = (S^3,L)\setminus(\nu(\frak b),\nu(\frak b\cap K))$ is the split union of a trivial tangle and an unlink, and $\Sigma$ gives rise to a bridge splitting of $(B^3,T)$.
	As before, this splitting is either minimal or perturbed. The case that the splitting is perturbed has the same consequence as in Case 1 above.
	
	If the splitting is minimal, then it is a split union of a 2--bridge splitting of the trivial tangle and a $(b-2)$--bridge splitting of an unlink.
	It follows that the bridge trisection is a split union: $\TT = \TT'\sqcup\TT''$, where $\TT'$ is a $(2,1)$--bridge trisection (of a projective plane, necessarily), and $\TT''$ is a $(b-2; c_1-1, c_2-1, b-2)$--bridge trisection (of an unlink of 2--spheres, necessarily).
	The latter is completely decomposable by Proposition~\ref{prop:c=b}.
\end{proof}

We can also use Theorem~\ref{thm:c=b} to understand surface-links with particular banded link presentations, where a \emph{banded link presentation} $(L,v)$  consists of an unlink $L \subset S^3$ and a collection of bands $v$ such that the resolution $L_v$ of $L$ along $v$ is also an unlink.  Every banded link presentation gives rise to a surface $\Ss$ in $S^4$, and conversely, every surface-link $\Ss$ in $S^4$ can be presented by a banded link~\cite{kss}.

In \cite[Section 3]{MeiZup_bridge1}, the authors introduced the notion of  \emph{banded bridge splitting} of $(L,v)$, a bridge splitting of $L$ such that the bands $v$ are isotopic into the bridge sphere with the surface framing and are dual to a collection of bridge disks on one side.  They showed that $(S^4,\Ss)$ admits a $(b;\mathbf{c})$--bridge trisection if and only if a banded link presentation $(L,v)$ of $\Ss$ admits a banded $b$-bridge splitting such that $|L| = c_1$, $|v| = b - c_2$, $|L_v| = c_3$.  As a corollary to Theorem~\ref{thm:c=b}, we obtain the following, which states, in essence, that a surface is unknotted if the bands are attached in a relatively simple way to the maxima or minima disks.

\begin{corollary}
Suppose a surface-link $\Ss$ in $S^4$ is presented by a banded link $(L,v)$ with a banded $b$--bridge splitting such that $b = |L| + 1$ or $b = |L_v| + 1$.  Then $\Ss$ is an unlink of 2--spheres or an unlink of 2--spheres and an unknotted projective plane.
\end{corollary}

The corollary exploits a feature of trisection theory called \emph{handle triality}:  If $(L,v)$ admits a banded bridge splitting as in the corollary, then it admits a $(b,\mathbf{c})$--bridge trisection such that $c_1 = b-1$ or $c_3 = b-1$.  By the three-fold symmetry of the trisection setup, we can extract a different banded link presentation with a single band, as in the proof of Theorem~\ref{thm:c=b}, and now we rely on known results about surface-links built with a single band to classify $\Ss$.  The result can be interpreted as an analog for knotted surfaces of Corollary~1.3 from~\cite{MSZ}.

\bibliographystyle{amsalpha}
\bibliography{BT-facts}

\end{document}